\documentclass[12pt,reqno]{amsart}
\usepackage[utf8]{inputenc}
\usepackage{enumerate}
\usepackage{mathtools}
\usepackage{mathrsfs}
\usepackage{amsmath,amssymb,xpatch,relsize,graphics,graphicx}

\newcommand{\di}{i}

\xapptocmd\normalsize{%
	\abovedisplayskip=12pt plus 3pt minus 9pt
	\abovedisplayshortskip=0pt plus 3pt
	\belowdisplayskip=12pt plus 3pt minus 9pt
	\belowdisplayshortskip=7pt plus 3pt minus 4pt
}{}{}
\theoremstyle{definition}
\newtheorem{definition}{Definition}[section]
\newtheorem{remark}[definition]{Remark}

\theoremstyle{plain}
\newtheorem{theorem}[definition]{Theorem}
\newtheorem{corollary}[definition]{Corollary}
\newtheorem{lemma}[definition]{Lemma}

\numberwithin{equation}{section}

\title[Differential Subordination]{\textbf{Differential Subordinations for functions with positive real part using Admissibility conditions}}
\author[M. Sharma]{Meghna Sharma}
\address{Department of Mathematics, University of Delhi, Delhi--110 007, India}
\email{meghnasharma203@gmail.com}
\author [S. Kumar]{Sushil Kumar}
\address {Bharati Vidyapeeth's college of Engineering, Delhi-110063, India}
\email{sushilkumar16n@gmail.com}
\author [N.K. Jain]{Naveen Kumar Jain}
\address {Department of Mathematics, Aryabhatta College, Delhi-110021,India}
\email{naveenjain05@gmail.com}

\date{}

\keywords{Differential subordination; admissibilty condition; modified sigmoid function; Janowski function; Exponential function}

\subjclass[2010]{30C45, 30C50}

\thanks{The first author is supported by Junior Research Fellowship from Council of Scientific and Industrial Research, New Delhi, Ref. No.:1753/(CSIR-UGC NET JUNE, 2018).}

\allowdisplaybreaks

\begin{document}

\maketitle

\begin{abstract}
Some sufficient conditions on certain constants which are  involved in some first, second and third order differential subordinations associated with certain functions with positive real part like modified Sigmoid function, exponential function and Janowski function are obtained so that the analytic function $p$ normalized by the condition $p(0)=1$,  is subordinate to Janowski function. The admissibility conditions for Janowski function
are used as a tool in the proof
of the results. As application, several
sufficient conditions are also computed for Janowski starlikeness.

\end{abstract}

\section{Introduction}
Let $\mathcal{A}$ denote the class of all analytic functions $f$ on the open unit disc $\mathbb{D}:= \{z \in \mathbb{C}: |z| < 1\}$ normalized by the conditions $f(0)=0$ and $f'(0)=1$. An analytic function defined on $\mathbb{D}$ is univalent if $f$ is one-to-one in $\mathcal{A}$.
Let $\mathcal{S} \subset \mathcal{A}$ be the class of all univalent functions. Denote the class of all analytic functions $f$ having Taylor series expansion $f(z)=a+a_{n}z^{n}+a_{n+1}z^{n+1}+...$, for some $a \in \mathbb{C}$ and fixed integer $n$ by $\mathcal{H}[a,n]$.
Let $f$ and $g$ be analytic in $\mathbb{D}$.
The function $f$ is subordinate to $g$, and write $f \prec g$,
if there exists an analytic function $w : \mathbb{D} \rightarrow \mathbb{D}$
with $|w(z)| \leq |z|$ such that $f(z) = g(w(z))$ for all $z \in \mathbb{D}$.
In particular, if $g \in \mathscr{U}$ then, $f \prec g$ if and only if
$f(0) = g(0)$ and $f(\mathbb{D}) \subseteq g(\mathbb{D})$.

Let $\mathcal{P}$ be the class of functions with positive real part of the form
$p(z)=1+c_1 z+c_2 z+\cdots$ over  $\mathbb{D}$.
Let  $A$ and  $B$ be  arbitrary fixed numbers which are satisfying the inequality $-1 \leq B < A \leq 1$, then the analytic function $p\in \mathcal{P}$ is known as the Janowski functions associated with right half plane  if it satisfies the subordination relation $p(z) \prec (1+Az)/(1+Bz)$ for all $z \in \mathbb{D}$.
The class of such functions is denoted by $\mathcal{P}[A,B]$. Let $\mathcal{S}^*[A,\,B]$ be the class of Janowski starlike functions $f\in \mathcal{A}$ such that
${z f'(z)}/{f(z)}\in \mathcal{P}[A,\, B]$ for $z \in \mathbb{D},$
introduced by Janowski \cite{MR0267103}.
Let $\mathcal{S}^*[A,B]$ be the class of the functions $f \in \mathcal{A}$ such that the quantity $zf'(z)/f(z)$ lies in the region $\Delta = \{w \in \mathbb{C}: |(w-1)/(A-Bw)| < 1\}$.
As a special case, we note that  $S^*[1-2 \alpha, -1] = \mathcal{S}^*(\alpha)$ that contains starlike functions of order $\alpha$ \cite{goodman 1,Robertson}.   
In 2015, authors \cite{MR3394060} introduced the class $\mathcal{S}^*_{e}$ which contains the functions $f\in \mathcal{A}$ satisfying the subordination relation ${zf'(z)}/{f(z)} \prec e^z$ for all $z\in \mathbb{D}$. In addition, if $f\in \mathcal{S}^*_{e}$, then the quantity $zf'(z)/f(z)$ lies in the domain $\{w\in\mathbb{C}\colon |\log w|<1\}$.
Recently, Goel and Kumar \cite{MR4044913} introduced and studied the class $\mathcal{S}^{*}_{SG}$ which contains starlike functions associated with modified sigmoid function $\phi_{SG}(z):=2/(1+e^{-z})$ and satisfy the subordination relation $zf'(z)/f(z) \prec \phi_{SG}$ for all $z \in \mathbb{D}$. In similar way, if the function $f\in \mathcal{S}^{*}_{SG}$, then the quantity $zf'(z)/f(z)$ lies in the domain $\{w\in\mathbb{C}\colon |\log w/(2-w)|<1\}$. For details,  see\cite{3,Naveen10,Naveen12}.

Goluzin \cite{Goluzin} studied initially the  first order differential subordination $zp'(z) \prec zq'(z)$, whenever $zq'(z)$ is convex,  the subordination $p\prec q$ holds   and the  function $q$ is the best dominant. After this  basic result, many authors established several generalizations of  differential subordination implications. In 1981, an article titled "Differential subordination and univalent functions" by Miller and Mocanu \cite{MR1760285} commenced the study of differential subordination as a generalized version of differential inequalities. For more details, see \cite{Ant94,Liu,Ozk,Miller03}.
In 1989, Nunokawa \emph{et al.\@} \cite{MR0975653} studied the first order differential subordination and proved that $1+zp'(z) \prec 1+z$ implies $p(z) \prec 1+z$.
They used this result to provide a criterion so that a normalized analytic function is univalent in $\mathbb{D}$.
Then, Ali \emph{et al.\@} \cite{MR2336133} generalized this result and proved that $p(z)$ is subordinate to the Janowski function whenever $1+\beta z p'(z)/p^{j}(z) \prec (1+Dz)/(1+Ez)$ for $j=0,1,2$.
Here, $A,B,D,E \in [-1,1]$.
Further, Ali \emph{et al.\@} \cite{MR2917253} determined the estimate on $\beta$ so that the subordination $1+\beta z p'(z)/p^{j}(z)$ is subordinate to the function $ \sqrt{1+z}$, $(j=0,1,2)$ which implies that $p(z)$ is subordinate to $\sqrt{1+z}$.
Later, Kumar \emph{et al.\@} \cite{MR3063215} computed a bound on $\beta$ so that $p(z) \prec \sqrt{1+z}$, whenever $1+\beta z p'(z)/p^{j}(z) \prec (1+Dz)/(1+Ez)$, $(j=0,1,2)$ with $|D| \leq 1$ and $-1 <E <1$.
Some of these results were not sharp.
Also, it was difficult to establish analogous results for certain functions with positive real parts such as $\phi_{0}(z):=1+(z/k)((k+z)/(k-z))$, $(k=1+\sqrt{2})$, $\phi_{\sin}(z):=1+\sin z$, $\mathcal{Q}(z):=e^{e^{z}-1}$ by the approach used in above discussed research work.
Later in 2018, Kumar and Ravichandran \cite{MR3800966} used some different approach and were able to established best possible bounds on $\beta$ so that  $1+\beta z p'(z)/p^{j}(z) $ is subordinate to $\sqrt{1+z}, (1+Az)/(1+Bz)$ which implies that $p(z) \prec e^{z}, (1+Az)/(1+Bz)$.
In 2018, Ahuja \emph{et al.\@} \cite{Ahuja18} also obtained sharp subordination implications results for the functions associated with lemniscate of Bernoulli. For  recent work related to first order differential subordinations, reader may refer \cite{Bohra19,Cho18,Ebadian20,Gandhi18,Kanika20,Wani20}.

In 2018, Madaan \emph{et al.\@} \cite{MR4036351} established first and second order differential subordinations associated with the lemniscate of Bernoulli using admissibility technique.
Further, Anand \emph{et al.\@} \cite{Anand19} also studied the generalized  first order differential subordination for the Janowski functions.
In 2019, Dorina R\u{a}ducanu \cite{MR3942345} established second order differential subordination implications associated with generalized Mittag-Leffler function.
For related work, readers may see \cite{MR4141599,MR3962536,Kanas06}.

Motivated by the aforesaid work, using admissibility conditions for Janowski functions, we determine certain conditions on $\beta$, $\gamma$, $A$ and $B$ where $-1 \leq B < A \leq 1$ so that $p$ belongs to the class $\mathcal{P}[A,B]$ whenever $1+\beta {z p^{\prime}(z)}/{p^{k}(z)}$, $1 + \beta {(zp'(z))^2}/{p^{k}(z)}$, $(1-\alpha) p(z) + \alpha p^2(z) + \beta z {p'(z)}/{p^{k}(z)}$ $(\alpha \in [0,\,1])$, $\left({1}/{p(z)}\right) - \beta z {p'(z)}/{p^{k}(z)}$, $p(z)+{zp'(z)}/{(\beta p(z)+ \gamma)^{k}}$ $(\gamma>0)$, $1+\gamma z p^{\prime}(z)+ \beta z^2 p''(z)$ and $p(z) +\gamma z p^{\prime}(z)+ \beta z^2 p''(z)$  are subordinate to some functions with positive real part like $e^z$, $2/(1+e^{-z})$ and $(1+Az)/(1+Bz)$, where $k$ is a positive integer.
Certain implications of these results are also discussed which gives sufficient conditions for an analytic function $f$ to be in the class $\mathcal{S}^{*}[A,B]$.

\section{The Admissibility Condition}
This section provides some basic facts related to admissibility conditions associated with Janowski function that will be needed in the proving our main results.
\begin{definition}
Let $\psi(r,s,t;z) : \mathbb{C}^3 \times \mathbb{D} \to \mathbb{D}$ be  analytic and  $h \in \mathcal{S}$.
Then a function $p \in \mathcal{A}$, satisfying following  differential subordination relation
\[\psi(p(z),zp'(z),z^2p''(z);z) \prec h(z)\]
is called its \emph{solution}.
\end{definition}

Let $\mathfrak{Q}$ denote the class of all  functions $q\in \mathcal{S}$ defined on $\overline{\mathbb{D}} \setminus \mathbf{E}(q)$, where
$\mathbf{E} = \{\zeta \in \partial \mathbb{D} : \lim\limits_{z \rightarrow \zeta} q(z) = \infty \}$	
such that $q'(\zeta) \neq 0$ for $\zeta \in \partial \mathbb{D} \setminus \mathbf{E}(q)$.

\begin{definition}
Let $\Omega \subset \mathbb{C}, q \in \mathfrak{Q}$ and $n \geq 1$.
Consider the class of admissible functions $\Psi_{n}[\Omega,q]$, consists of those functions
$\psi: \mathbb{C}^3 \times \mathbb{D} \to \mathbb{C}$
which satisfy the admissibility condition:
\[\psi(r,s,t;z) \notin \Omega\]
whenever
\[r=q(\zeta),
s=m \zeta q'(\zeta)\,\, \text{and}\,\,
\operatorname{Re}\left(\frac{t}{s}+1\right) \geq m \operatorname{Re}\left(\frac{\zeta q''(\zeta)}{q'(\zeta)}+1\right)\]
for $z \in \mathbb{D}$,
$\zeta \in \partial \mathbb{D} \setminus \mathbf{E}(q)$ and $m\geq n\geq1$.

We write the class $\Psi_{1}[\Omega,q]$ as $\Psi[\Omega,q]$.
\end{definition}

%
\begin{theorem} \cite[Theorem 2.3b, p.28]{MR1760285}
Let the function $\psi \in \Psi_n[\Omega,q]$ and $q(0)=a$.
If $p \in \mathcal{H}[a,n]$,  then
\begin{equation}\label{second order implication}
\psi(p(z), zp'(z), z^{2}p''(z); z) \in \Omega \implies p(z) \prec q(z).
\end{equation}
\end{theorem}

Let $\Omega$ be a simply connected domain which is not the entire complex plane, then there exists a conformal mapping $h$ from $\mathbb{D}$ onto $\Omega$ with $h(0)=\psi(a,0,0;0)$.
Therefore, if $p \in \mathcal{H}[a,n]$, equation \eqref{second order implication} can be written as
\begin{equation}\label{admissibility eqn 2}
\psi(p(z), zp'(z), z^{2}p''(z); z) \prec h(z) \implies p(z) \prec q(z).
\end{equation}
The univalent function $q$ is called \emph{dominant of the solutions} of the differential subordination \eqref{admissibility eqn 2}.
The function $\tilde{q}$ is called the \emph{best dominant} of \eqref{admissibility eqn 2} if $\tilde{q} \prec q$ for all dominants of \eqref{admissibility eqn 2}. \\

Consider the function $q(z) = (1+Az)/(1+Bz)$ for $-1 \leq A < B \leq 1$.
Denote the class $\Psi_{n}[\Omega,(1+Az)/(1+Bz)]$ by $\Psi_{n}[\Omega;A,B]$.
Therefore, the admissibility conditions for the function
$q$ are given as follows:

\begin{theorem}\label{second order admissibility theorem}\cite{Anand19}
Let the function $p \in \mathcal{H}[1,n]$ such that $p(z) \not\equiv 1$ and $n \geq 1$ and  $\Omega$ be a subset of  $\mathbb{C}$.
The class $\Psi_{n}[\Omega;A,B]$ is defined as the class of all those functions $\psi:\mathbb{C}^{3} \times \mathbb{D} \rightarrow \mathbb{C}$
such that
\[\psi(r,s,t;z) \not\in \Omega \quad \text{whenever} \quad (r,s,t;z) \in \operatorname{Dom} \psi \quad \text{and} \]
\[r = q(\zeta) = \frac{1+Ae^{i \theta}}{1+Be^{i \theta}},
s = m \zeta q'(\zeta)=\frac{m(A-B)e^{i \theta}}{(1+Be^{i \theta})^2}
\text{ and }
\operatorname{Re}\left(\frac{t}{s} + 1\right) \geq \frac{m(1-B^2)}{1+B^2+2B\cos\theta}\]
for $z \in \mathbb{D}, \theta \in (0, 2\pi)$ and $m \geq 1$.
\end{theorem}
On taking $\psi \in \Psi_{n}[\Omega;A,B]$ in Theorem \eqref{second order admissibility theorem}, we have
\begin{corollary}
If $(p(z),zp'(z),z^2p''(z);z) \in \operatorname{Dom} \psi$ and $\psi(p(z),zp'(z),z^2p''(z);z) \in \Omega$ for $z\in \mathbb{D}$, then $p$ belongs to $ \mathcal{P}[A,B]$.
\end{corollary}

For $\zeta=e^{i \theta}$, where $\theta \in [0,2\pi)$,
let us consider
\begin{equation}\label{k(theta)}
|q(\zeta)| =\sqrt{\frac{1+A^2+2A\cos\theta}{1+B^2+2B\cos\theta}}:=k(\theta)
\end{equation}
whose minimum value is $\frac{1-A}{1-B}$, attained at $\theta=\pi$.
Also, observe that
\begin{equation}\label{d(theta)}
|q'(\zeta)|= \frac{A-B}{1+B^2+2B\cos\theta}:=d(\theta)
\end{equation}
and the minimum value of $d(\theta)$ is $d(0)=\frac{A-B}{1+B^2+2B}$ for $B>0$ and $d(\pi)=\frac{A-B}{1+B^2-2B}$ for $B<0$.
Note that
\begin{equation}\label{g(theta)}
\operatorname{Re}\left(\frac{\zeta q''(\zeta)}{q'(\zeta)}\right) = \frac{-2B(B+\cos\theta)}{1+B^2+2B\cos\theta}:=g(\theta)
\end{equation}
and the minimum value of $g(\theta)$ is $g(0)=\frac{-2B(B+1)}{1+B^2+2B}$ for $B>0$ and $g(\pi)=\frac{-2B(B-1)}{1+B^2-2B}$ for $B<0$.

Using above values, we get the admissibility condition for third order differential subordination as follows:

%

In order to the prove our main results, we will use the following lemmas extensively.
\begin{lemma}\label{lemma for e^{z}}\cite{MR4141599}
	Let $z$ be a complex number. Then
	\[|\log(1+z)| \geq 1 \quad \text{if and only if} \quad |z| \geq e-1.\]
\end{lemma}

\begin{lemma}\label{lemma for admissibility}
	Consider the disc $\Delta_{\beta} = \{w \in \mathbb{C} :|w|<\beta, 0< \beta \leq 1\}$.
	Then the inequality
	\[\left|\log \left(\frac{w}{2-w}\right)\right| \geq 1\]
	holds if and only if
	$|w| \geq \beta_{0} \approx 0.473519$, where $\beta_{0}$ is the positive real root of the equation $(e^2-1)\beta^4-2(e^2-4)\beta^3+4(e^2-6)\beta^2+32\beta-16=0$.
\end{lemma}

\begin{proof}
For $\theta \in [0,2\pi]$, let $w=\beta e^{i \theta}$ be a boundary point of the disc $\Delta_{\beta}$.
Consider
\begin{align*}
\left|\log \left(\frac{\beta e^{i \theta}}{2-\beta e^{i \theta}}\right)\right|^2
&=\left|\frac{1}{2}\log\left(\frac{4\beta^{2}+\beta^{4}-4\beta^{3}\cos\theta}{(4+\beta^{2}-4\beta\cos\theta)^2}\right)+\di \arctan \left(\frac{2\beta\sin{\theta}}{2\beta\cos{\theta} - \beta^{2}}\right)\right|^2\\
&=\left(\log\left(\frac{\sqrt{4\beta^2+\beta^4-2\beta^3\cos\theta}}{4+\beta^{2}-4\beta\cos\theta}\right)\right)^2 + \left(\arctan \left(\frac{2\sin{\theta}}{2\cos{\theta} - \beta}\right)\right)^2\\
&:=f(\beta,\theta)
\end{align*}
In the interval $[0,2\pi]$, the function $f(\beta,\theta)$ attains its absolute minimum at $\theta=0$ and therefore,
$f(\beta,\theta) \geq f(\beta,0)$ for all $\theta \in [0,2\pi]$.
Thus, the inequality
\[\left|\log \left(\frac{\beta e^{i \theta}}{2-\beta e^{i \theta}}\right)\right| \geq 1\]
holds if and only if
\[f(\beta,0) \geq 1\]
or equivalently,
\begin{equation}\label{lemma inequality}
(e^2-1)\beta^4-2(e^2-4)\beta^3+4(e^2-6)\beta^2+32\beta-16 \geq 0.
\end{equation}
Therefore, by Intermediate Value Theorem, the inequality\eqref{lemma inequality} holds for $\beta \geq \beta_{0} \approx 0.473519$ which is the positive real root of the equation in \eqref{lemma inequality}.
\end{proof}

\section{First Order Differential Subordination}
In order to prove first order differential subordination relations, we need the following result due to Swati \emph{et al.\@}~\cite{MR3800966}, which is a specific case of Theorem \eqref{second order admissibility theorem}.
\begin{theorem}\label{first order admissibility thm}
Let $p \in \mathcal{H}[1,n]$ such that $p(z) \not\equiv 1$ and $n \geq 1$.
Let $\Omega$ be a set in $\mathbb{C}$.
The class $\Psi_{n}[\Omega;A,B]$ is defined as the class of all those functions $\psi:\mathbb{C}^2 \times \mathbb{D} \to \mathbb{C}$
such that
\[\psi(r,s;z) \not\in \Omega \quad \text{whenever} \quad (r,s;z) \in \operatorname{Dom} \psi \quad \text{and} \]
\begin{equation}\label{r and s}
r = \frac{1+Ae^{i \theta}}{1+Be^{i \theta}} \quad \text{ and } \quad
s = \frac{m(A-B)e^{i \theta}}{(1+Be^{i \theta})^2}
\end{equation}
for $z \in \mathbb{D}, \theta \in (0, 2\pi)$ and $m \geq 1$.
\end{theorem}
Consequently, when $\psi \in \Psi_{n}[\Omega;A,B]$, the above theorem becomes:
Let $z \in \mathbb{D}$. If $(p(z),zp'(z);z) \in \operatorname{Dom} \psi$ and $\psi(p(z),zp'(z);z) \in \Omega$, then $p \in \mathcal{P}[A,B]$.\\

Using above theorem, we determine the conditions on $A$, $B$ and $\beta$ so that the function $\psi(p(z),zp'(z);z)$ is subordinate to Modified Sigmoid function $\phi_{SG}$ and exponential function $e^{z}$ implies $p(z)$ is subordinate to $(1+Az)/(1+Bz)$.

\begin{theorem}\label{admissibility first thm}
Let $-1 \leq B < A \leq 1$ and $k$ be a non-negative integer. Let $p$ be an analytic function defined on $\mathbb{D}$, which satisfies $p(0)=1$ and $\beta_{0} \approx 0.475319$.
Then the following are sufficient for $p \in \mathcal{P}[A,B]$.
\begin{enumerate}[(a)]
\item $1+\beta \frac{z p^{\prime}(z)}{p^{k}(z)} \prec \phi_{SG}$, where $|\beta| \geq
\begin{cases}
\frac{\beta_{0}(1+|A|)^{k}(1+|B|)^{2-k}}{(A-B)}, & \text{ when } 0\leq k \leq 2 \\
\frac{\beta_{0} (1+|A|)^k}{(A-B)(1-|B|)^{k-2}}, & \text{ when } k>2
\end{cases}$.
\item $1+\beta \frac{z p^{\prime}(z)}{p^{k}(z)} \prec e^{z}$, where
$|\beta| \geq
\begin{cases}
\frac{(e-1)(1+|A|)^{k}(1+|B|)^{2-k}}{(A-B)}, & \text{ when } 0\leq k\leq 2 \\
\frac{(e-1)(1+|A|)^{k}}{(A-B)(1-|B|)^{k-2}}, & \text{ when } k>2
\end{cases}$.
\end{enumerate}
\end{theorem}

\begin{proof}
\begin{enumerate}[(a)]
\item Let $\Omega = \phi_{SG}(\mathbb{D}) = \{w \in \mathbb{C}: \left|\log \left(\frac{w}{2-w}\right)\right| < 1\}$.
Consider the analytic function
$\psi:\mathbb{C}\setminus \{0\} \times \mathbb{C} \times \mathbb{D} \rightarrow \mathbb{C}$ defined as
\[\psi(r,s;z) = 1+\beta \frac{s}{r^k}.\]
In accordance with Theorem \eqref{first order admissibility thm},
$\psi \in \Psi[\Omega;A,B]$, if $\psi(r,s;z) \notin \Omega$
where
$r$ and $s$ are given in the equation \eqref{r and s}.
Therefore, it is enough to show that the required subordination holds if
\begin{equation}\label{log inequality first order MSF}
\left|\log \left(\frac{\psi(r,s;z)}{2-\psi(r,s;z)}\right)\right| \geq 1.
\end{equation}

\begin{enumerate}[(i)]
\item When $0 \leq k \leq 2$, let us consider
\begin{align*}
|\psi(r,s;z)| &= \left| 1 + \beta \frac{m(A-B) e^{i \theta}}{(1+A e^{i \theta})^k(1+B e^{i \theta})^{2-k}}\right|\\
&= \left|\frac{(1+A e^{i \theta})^k(1+B e^{i \theta})^{2-k} + \beta m (A-B) e^{i \theta}}{(1+A e^{i \theta})^k(1+B e^{i \theta})^{2-k}}\right|\\
&\geq \frac{|\beta|m(A-B)-|(1+A e^{i \theta})^k||(1+B e^{i \theta})^{2-k}|}{|(1+A e^{i \theta})^k||(1+B e^{i \theta})^{2-k}|}\\
&\geq \frac{|\beta|m(A-B)-(1+|A|)^k(1+|B|)^{2-k}}{(1+|A|)^k(1+|B|)^{2-k}}\\
&=:\phi(m).
\end{align*}
By First Derivative Test, $\phi(m)$ is an increasing function for $m \geq 1$.
This implies $\phi(m) \geq \phi(1)$ for all $m \geq 1$.
Hence, the last inequality reduces to
\[|\psi(r,s;z)| \geq \phi(1)\]
where
\[|\phi(1)|= \frac{|\beta|(A-B)-(1+|A|)^k(1+|B|)^{2-k}}{(1+|A|)^k(1+|B|)^{2-k}}.\]
Using Lemma \eqref{lemma for admissibility}, the inequality \eqref{log inequality first order MSF}
is true if
\[|\phi(1)| \geq \beta_{0}\]
or
\[\frac{|\beta|(A-B)-(1+|A|)^k(1+|B|)^{2-k}}{(1+|A|)^k(1+|B|)^{2-k}} \geq \beta_{0}.\]
or
\[|\beta|(A-B)-(1+|A|)^k(1+|B|)^{2-k} \geq \beta_{0} (1+|A|)^k(1+|B|)^{2-k},\]
which shows that $\psi(r,s;z) \notin \Omega$ for
$|\beta| \geq \beta_{0}\frac{(1+|A|)^k(1+|B|)^{2-k}}{(A-B)}.$

\item When $k > 2$, observe that
\begin{align*}
|\psi(r,s;z)| &= \left| 1 + \beta \frac{m(A-B)(1+B e^{i \theta})^{k-2} e^{i \theta}}{(1+A e^{i \theta})^k}\right|\\
&= \left|\frac{(1+A e^{i \theta})^k + \beta m (A-B)(1+B e^{i \theta})^{k-2} e^{i \theta}}{(1+A e^{i \theta})^k}\right|\\
&\geq \frac{|\beta|m(A-B)|(1+B e^{i \theta})^{k-2}|-|(1+A e^{i \theta})^k|}{|(1+A e^{i \theta})^k|}\\
&\geq \frac{|\beta|m(A-B)(1-|B|)^{k-2}-(1+|A|)^k}{(1+|A|)^k}\\
&=:\phi(m)
\end{align*}
As in the earlier case, note that $\phi(m)$ is an increasing function of $m$.
Hence, $\phi(m) \geq \phi(1)$ for all $m \geq 1$.
By Lemma \eqref{lemma for admissibility}, \eqref{log inequality first order MSF} holds if
\[|\psi(r,s;z)| \geq |\phi(1)| \geq \beta_{0}\]
where \[\phi(1) = \frac{|\beta| (A-B)(1-|B|)^{k-2}-(1+|A|)^k}{(1+|A|)^k}.\]
Therefore, for $|\beta| \geq \beta_{0}\frac{(1+|A|)^k}{(A-B)(1-|B|)^{k-2}}$,
$\psi \in \Psi[\Omega;A,B]$  and hence, $p \in \mathcal{P}[A,B]$.
\end{enumerate}

\item Consider the domain $\Omega = \{w \in \mathbb{C}: |\log w| < 1\}$.
Let $\psi:\mathbb{C}\backslash \{0\} \times \mathbb{C} \times \mathbb{D} \rightarrow \mathbb{C}$ be defined by
\[\psi(r,s;z) = 1+\beta \frac{s}{r^k}.\]
By Theorem \eqref{first order admissibility thm}, $\psi$ belongs to $\Psi[\Omega;A,B]$ if $\psi(r,s;z) \notin \Omega$ for $z \in \mathbb{D}$.
This implication holds  if
\begin{equation}\label{log inequality first order exp}
|\log(\psi(r,s;z))| \geq 1
\end{equation}
Since
\begin{align*}
|\log(\psi(r,s;z))| &= \left|\log\left(1+\beta \frac{s}{r^k}\right)\right|\\
&= \left|\log\left(1+\beta \frac{m(A-B)e^{i \theta}(1+Be^{i \theta})^{k}}{(1+Be^{i \theta})^{2}(1+Ae^{i \theta})^{k}}\right)\right|,\\
\end{align*}
by Lemma\eqref{lemma for e^{z}}, inequality \eqref{log inequality first order exp} holds if and only if
\begin{equation}\label{condition for e^z first order}
\left|\frac{\beta m(A-B)e^{i \theta}(1+Be^{i \theta})^{k}}{(1+Be^{i \theta})^{2}(1+Ae^{i \theta})^{k}}\right| \geq e-1.
\end{equation}

\begin{enumerate}[(i)]
\item When $0 \leq k \leq 2$ , consider
\begin{align*}
\left|\frac{\beta m(A-B)e^{i \theta}}{(1+Be^{i \theta})^{2-k}(1+Ae^{i \theta})^{k}}\right|
& \geq \frac{|\beta|m(A-B)}{|(1+Be^{i \theta})^{2-k}||(1+Ae^{i \theta})^{k}|}\\
& \geq \frac{|\beta|(A-B)}{(1+|B|)^{2-k}(1+|A|)^{k}} \quad (\because m \geq 1)\\ & \geq e-1
\end{align*}
if
\[|\beta| \geq \frac{(e-1)(1+|A|)^{k}(1+|B|)^{2-k}}{(A-B)}.\]
which shows that $\psi(r,s;z) \notin \Omega$ and hence the required result holds.

\item When $k > 2$, observe that
$$
\begin{aligned}
\left|\frac{\beta m(A-B)e^{i \theta}(1+Be^{i \theta})^{k-2}}{(1+Ae^{i \theta})^{k}}\right|
& \geq \frac{|\beta|m(A-B)|(1+Be^{i \theta})^{k-2}|}{|(1+Ae^{i \theta})^{k}|}\\
& \geq \frac{|\beta|(A-B)(1-|B|)^{k-2}}{(1+|A|)^{k}} \quad (\because m \geq 1)
\end{aligned}
$$
Therefore, \eqref{condition for e^z first order} holds if
\[\frac{|\beta|(A-B)(1-|B|)^{k-2}}{(1+|A|)^{k}} \geq e-1\]
or equivalently,
\[|\beta| \geq \frac{(e-1)(1+|A|)^{k}}{(A-B)(1-|B|)^{k-2}}.\]
which proves that $\psi(r,s;z) \notin \Omega$ and hence, $p\in \mathcal{P}[A,B]$.
\end{enumerate}
\end{enumerate}
\end{proof}

\begin{corollary}\label{first corollary}
Let $f \in \mathcal{A}$ and $\beta_{0} \approx 0.473519$.
Set
$\mathcal{G}(z):= \frac{f'(z)}{f(z)} - z\left(\frac{f'(z)}{f(z)}\right)^{2} + \frac{zf''(z)}{f(z)}$.
If one of the following subordination holds, then $f \in S^{*}[A,B].$
\begin{enumerate}[(a)]
\item $1+\beta z \mathcal{G}(z) \prec \phi_{SG}$ for $|\beta|(A-B) \geq \beta_{0} (1+|B|)^{2}$,
\item $1+ \beta \left(\frac{f(z)}{f'(z)}\right) \mathcal{G}(z) \prec \phi_{SG}$ for
$|\beta|(A-B) \geq \beta_{0}(1+|A|)(1+|B|)$,
\item $1+ \frac{\beta}{z} \mathcal{G}(z)  \prec \phi_{SG}$ for
$|\beta|(A-B) \geq \beta_{0}(1+|A|)^{2}$,
\item $1+\beta z \mathcal{G}(z) \prec e^{z}$ for $|\beta|(A-B) \geq (e-1) (1+|B|)^{2}$,
\item $1+ \beta \left(\frac{f(z)}{f'(z)}\right) \mathcal{G}(z) \prec e^{z}$ for
$|\beta|(A-B) \geq (e-1)(1+|A|)(1+|B|)$,
\item $1+ \frac{\beta}{z} \mathcal{G}(z)  \prec e^{z}$ for
$|\beta|(A-B) \geq (e-1)(1+|A|)^{2}$.
\end{enumerate}	
\end{corollary}

\begin{theorem}
Let $-1 \leq B < A \leq 1$, $\beta_{0} \approx 0.473519$, $k$ be a non-negative integer and $p \in \mathcal{A}$ such that $p(0)=1$.
If any of the following subordinations holds true, then
$p(z) \in \mathcal{P}[A,B]$.
\begin{enumerate}[(a)]
\item $1 + \beta \frac{(zp'(z))^2}{p^{k}(z)} \prec \phi_{SG}$,
where
$|\beta| \geq
\begin{cases}
\frac{\beta_{0}(1+|A|)^{k}(1+|B|)^{4-k}}{(A-B)^{2}}, & \text{ when } 0\leq k \leq 4 \\
\frac{\beta_{0}(1+|A|)^{k}}{(A-B)^{2}(1-|B|)^{k-4}}, & \text{ when } k>4
\end{cases}$
\item $1 + \beta \frac{(zp'(z))^2}{p^{k}(z)} \prec e^{z}$,
where
$|\beta| \geq
\begin{cases}
\frac{(e-1)(1+|A|)^{k}(1+|B|)^{4-k}}{(A-B)^{2}}, & \text{ when } 0\leq k \leq 4 \\
\frac{(e-1)(1+|A|)^{k}}{(A-B)^{2}(1-|B|)^{k-4}}, & \text{ when } k>4
\end{cases}$
\end{enumerate}
\end{theorem}

\begin{proof}
\begin{enumerate}[(a)]
\item Consider $\Omega$ as in Theorem \eqref{admissibility first thm}(a).
Define the analytic function $\psi:\mathbb{C}\setminus \{0\} \times \mathbb{C} \times \mathbb{D} \to \mathbb{C}$ as
\[\psi(r,s;z)=1+\beta \frac{s^2}{r^k}.\]
Therefore, we have
\[\psi(r,s;z) = 1 + \beta \frac{m^2(A-B)^2 e^{2i \theta}(1+B e^{i \theta})^k}{(1+A e^{i \theta})^k(1+B e^{i \theta})^4}.\]
Proceeding in the similar manner as in Theorem \eqref{admissibility first thm}, we have the following two cases.

\begin{enumerate}[(i)]
\item When $0 \leq k \leq 4$, let us consider
\begin{align*}
|\psi(r,s;z)|
&= \left|1 + \beta \frac{m^2(A-B)^2 e^{2i \theta}(1+B e^{i \theta})^k}{(1+A e^{i \theta})^k(1+B e^{i \theta})^4}\right|\\
&= \left|\frac{(1+A e^{i \theta})^k(1+B e^{i \theta})^{4-k} + \beta m^2(A-B)^2 e^{2i \theta}}{(1+A e^{i \theta})^k(1+B e^{i \theta})^{4-k}}\right|\\
&\geq \frac{|\beta| m^2(A-B)^2 - |(1+A e^{i \theta})^k||(1+B e^{i \theta})^{4-k}|}{|(1+A e^{i \theta})^k||(1+B e^{i \theta})^{4-k}|}\\
&\geq \frac{|\beta| m^2(A-B)^2 - (1+|A|)^{k} (1+|B|)^{4-k}}{(1+|A|)^{k} (1+|B|)^{4-k}}\\
&=:\phi(m)
\end{align*}
Simple observation shows that $\phi(m)$ in an increasing function for $m \geq 1$.
The required subordination result holds if $\psi(r,s;z) \notin \Omega$.
So, using Lemma \eqref{lemma for admissibility}, it is concluded that $p \in \mathcal{P}[A,B]$ if $|\beta| \geq \frac{\beta_{0}(1+|A|)^{k}(1+|B|)^{4-k}}{(A-B)^{2}}$.\\
	
\item When $k > 4$, observe that
\begin{align*}
|\psi(r,s;z)|
&= \left|1 + \beta \frac{m^2(A-B)^2 e^{2i \theta}(1+B e^{i \theta})^k}{(1+A e^{i \theta})^k(1+B e^{i \theta})^4}\right|\\
&= \left|\frac{(1+A e^{i \theta})^k + \beta m^2(A-B)^2 e^{2i \theta}(1+B e^{i \theta})^{k-4}}{(1+A e^{i \theta})^k}\right|\\
&\geq \frac{|\beta| m^2 (A-B)^2 |(1+B e^{i \theta})^{k-4}| - |(1+A e^{i \theta})^k|}{|(1+A e^{i \theta})^k|}\\
&\geq \frac{|\beta| m^2 (A-B)^2 (1-|B|)^{k-4} - (1+|A|)^{k}}{(1+|A|)^{k} }\\
&=:\phi(m)
\end{align*}
Noting that $\phi'(m)>0$ for $m \geq 1$ and proceeding  as in the part (i), we get the desired subordination result.
\end{enumerate}

\item Let $\Omega = \{w \in \mathbb{C}: |\log w| < 1\}$ be the domain.
Let $\psi:\mathbb{C}\backslash \{0\} \times \mathbb{C} \times \mathbb{D} \rightarrow \mathbb{C}$ be defined as
\[\psi(r,s;z) = 1+\beta \frac{s^2}{r^k}.\]
On the similar lines of the proof of Theorem \eqref{admissibility first thm} and using Lemma \eqref{lemma for e^{z}}, we get the desired result if
\begin{equation}\label{condition for e^z first order 2}
\left|\frac{\beta m^2(A-B)^2 e^{2i \theta}(1+B e^{i \theta})^k}{(1+A e^{i \theta})^k(1+B e^{i \theta})^4}\right| \geq e-1.
\end{equation}

\begin{enumerate}[(i)]
\item When $0 \leq k \leq 4$ , consider
\begin{align*}
\left|\frac{\beta m^2(A-B)^2 e^{2i \theta}}{(1+A e^{i \theta})^k(1+B e^{i \theta})^{4-k}}\right|
& \geq \frac{|\beta|m^2 (A-B)^2}{|(1+Ae^{i \theta})^{k}||(1+Be^{i \theta})^{4-k}|}\\
& \geq \frac{|\beta|(A-B)^2}{(1+|A|)^{k}(1+|B|)^{4-k}} \quad (\because m \geq 1)
\end{align*}
Now
\[\frac{|\beta|(A-B)^2}{(1+|A|)^{k}(1+|B|)^{4-k}} \geq e-1\]
if and only if
\[|\beta| \geq \frac{(e-1)(1+|A|)^{k}(1+|B|)^{4-k}}{(A-B)^{2}}.\]
	
\item When $k > 4$, observe that
\begin{align*}
\left|\frac{\beta m^2 (A-B)^2 e^{2 i \theta}(1+Be^{i \theta})^{k-4}}{(1+Ae^{i \theta})^{k}}\right|
& \geq \frac{|\beta|m^{2} (A-B)^2|(1-Be^{i \theta})^{k-4}|}{|(1+Ae^{i \theta})^{k}|}\\
& \geq \frac{|\beta|(A-B)^2(1-|B|)^{k-4}}{(1+|A|)^{k}} \quad (\because m \geq 1)\\
& \geq e-1.
\end{align*}

which yields the desired estimate on $\beta$.\qedhere
\end{enumerate}
\end{enumerate}
\end{proof}

\begin{corollary}
Let $f \in \mathcal{A}$ and $\mathcal{G}(z)$ be same as defined in Corollary \eqref{first corollary}.
Then each of the following subordination imply $f \in S^{*}[A,B]$.
\begin{enumerate}[(a)]
\item $1+\beta (z\mathcal{G}(z))^2 \prec \phi_{SG}$ for $|\beta|(A-B)^2 \geq \beta_{0}(1+|B|)^{4}$,
\item $1+ \beta z \left(\frac{f(z)}{f'(z)}\right) (\mathcal{G}(z))^2 \prec \phi_{SG}$ for $|\beta|(A-B)^2 \geq \beta_{0}(1+|A|)(1+|B|)^{3}$,
\item $1+ \beta \left(\frac{f(z)}{f'(z)}\right)^{2} (\mathcal{G}(z))^2 \prec \phi_{SG}$ for $|\beta|(A-B)^2 \geq \beta_{0}(1+|A|)^{2}(1+|B|)^{2}$,
\item $1+\beta (z\mathcal{G}(z))^2 \prec e^{z}$ for $|\beta|(A-B)^2 \geq (e-1)(1+|B|)^{4}$,
\item $1+ \beta z \left(\frac{f(z)}{f'(z)}\right) (\mathcal{G}(z))^2 \prec e^{z}$ for $|\beta|(A-B)^2 \geq (e-1)(1+|A|)(1+|B|)^{3}$,
\item $1+ \beta \left(\frac{f(z)}{f'(z)}\right)^{2} (\mathcal{G}(z))^2 \prec e^{z}$ for $|\beta|(A-B)^2 \geq (e-1)(1+|A|)^{2}(1+|B|)^{2}$.
\end{enumerate}
\end{corollary}

\begin{theorem}
Let $-1 \leq B < A \leq 1$,
$\beta_{0} \approx 0.475319$ and
$\alpha \in [0,1]$.
If $p \in \mathcal{P}$ satisfy the differential subordination
\[(1-\alpha) p(z) + \alpha p^2(z) + \beta z \frac{p'(z)}{p^{k}(z)} \prec \phi_{SG}(z),
\text{ where }\]
$$
|\beta| \geq
\frac{\beta_{0}(1+|A|)^{k} (1+|B|)^{2}+ (1 - \alpha){(1+|A|)^{k+1}(1+|B|)+\alpha (1 + |A|)^{k+2}}}{(A-B)(1-|B|)^{k}}.
$$

then $p(z) \prec (1+Az)/(1+Bz)$.
\end{theorem}

\begin{proof}
Let $\Omega$ be same as in Theorem \eqref{admissibility first thm}(a). Let the function $\psi$ be defined as
\[\psi(r,s;z) = (1-\alpha)r + \alpha r^2 +\beta \frac{s}{r^k}.\]
Substituting the values of $r$ and $s$ from equation\eqref{r and s}, we get
\[\psi(r,s;z) = (1 - \alpha)\frac{1+A e^{i \theta}}{1+ B e^{i \theta}} + \alpha \frac{(1+A e^{i \theta})^2}{(1+B e^{i \theta})^2} + \beta \frac{m (A-B) e^{i \theta} (1+B e^{i \theta})^k}{(1+A e^{i \theta})^k (1+Be^{i \theta})^2}\]

Then
\begin{align*}
|\psi(r,s;z)|
&= \left|(1 - \alpha)\frac{1+A e^{i \theta}}{1+ B e^{i \theta}} + \alpha \frac{(1+A e^{i \theta})^2}{(1+B e^{i \theta})^2} + \beta \frac{m (A-B) e^{i \theta} }{(1+A e^{i \theta})^k (1+B e^{i \theta})^{2-k}}\right|\\
&= \left|\frac{(1 - \alpha)(1+A e^{i \theta})^{k+1}(1+B e^{i \theta}) +\alpha (1+A e^{i \theta})^{k+2} + \beta m (A-B) e^{i \theta} (1+B e^{i \theta})^k}{(1+A e^{i \theta})^k (1+B e^{i \theta})^2}\right|\\
&\geq \frac{|\beta| m (A-B) (1-|B|)^{k}-(1-\alpha)|(1+A e^{i \theta})^{k+1}||(1+B e^{i \theta})|-\alpha |(1+A e^{i \theta})^{k+2}|}{|(1+A e^{i \theta})^k||(1+B e^{i \theta})^{2}|}\\
&\geq \frac{|\beta| m (A-B)(1-|B|)^{k}-(1 - \alpha)(1+|A|)^{k+1} (1+|B|) - \alpha (1 + |A|)^{k+2}}{(1+|A|)^{k} (1+|B|)^{2}}\\
&=:\phi(m)
\end{align*}
Verify that the function $\phi(m)$ is increasing $\forall$ $m \geq 1$
and hence, attains its minimum value at $m=1$. Since
\[\frac{|\beta| (A-B)(1-|B|)^{k} - (1 - \alpha)(1+|A|)^{k+1} (1+|B|) - \alpha (1 + |A|)^{k+2}}{(1+|A|)^{k} (1+|B|)^{2}} \geq \beta_{0}\]
by  Theorem \eqref{first order admissibility thm} and Lemma \eqref{lemma for admissibility}, we get the desired result.
	
\end{proof}

\begin{remark}
For $\alpha=0$, the above theorem reduces to the following result.
\end{remark}

\begin{corollary}
Let $p$ be an analytic function satisfying $p(0)=1$ and $\beta > \beta_{0} \approx 0.475319$.
Then each of the following subordinations is sufficient to imply $p \in \mathcal{P}[A,B]$.
\[
p(z)+\beta z \frac{p'(z)}{p^{k}(z)} \prec \phi_{SG},
\text { where }
|\beta| \geq
\frac{\beta_{0}(1+|A|)^{k} (1+|B|)^{2}+ {(1+|A|)^{k+1}(1+|B|)}}{(A-B)(1-|B|)^{k}}.
\]
\end{corollary}
	
\begin{theorem}
Let $\beta_{0} \approx0.475319$ and $k$ be a non-negative integer.
If $p \in \mathcal{P}$ and satisfies the differential subordination
\[
\left(\frac{1}{p(z)}\right) - \beta z \frac{p'(z)}{p^{k}(z)} \prec \phi_{SG}(z) ,
\text { where }
|\beta| \geq
\begin{cases}
\frac{\beta_{0}(1+|A|)^{k+1}(1+|B|)^{2-k}+(1+|A|)^{k}(1+|B|)^{3-k}}{(A-B)(1-|A|)}, & \text { when } 0 \leq k \leq 2 \\
\frac{\beta_{0}(1+|A|)^{k}+(1+|A|)^{k}(1+|B|)}{(A-B)(1-|A|)(1-|B|)^{k-2}}, & \text { when } k>2
\end{cases},
\]
 then $p(z) \prec (1+Az)/(1+Bz)$.
\end{theorem}

\begin{proof}
Let $\Omega$ be same as in Theorem\eqref{admissibility first thm}(a).
Consider the analytic function $\psi$ defined as
\[\psi(r,s;z) = \frac{1}{r} - \beta \frac{s}{r^k}.\]
Substituting the values of $r$ and $s$ as given in equation \eqref{r and s}, we get
\[\psi(r,s;z) = \frac{1 +Be^{i \theta}}{1 +Ae^{i \theta}} - \beta \frac{m(A-B)e^{i\theta}(1 +Be^{i \theta})^{k}}{(1 +Ae^{i \theta})^{k}(1 +Be^{i \theta})^2}\]
Proceeding as in Theorem \eqref{admissibility first thm}(a), the following two cases arises.
\begin{enumerate}[(i)]
\item When $0 \leq k \leq 2$, consider
\begin{align*}
|\psi(r,s;z)|&=\left|\frac{1 +Be^{i \theta}}{1 +Ae^{i \theta}} - \beta \frac{m(A-B)e^{i\theta}(1 +Be^{i \theta})^{k}}{(1 +Ae^{i \theta})^{k}(1 +Be^{i \theta})^2}\right|\\
&=\left|\frac{(1+Ae^{i\theta})^{k}(1+B e^{i \theta})^{3-k}-\beta m (A-B) e^{i \theta}(1+Ae^{i\theta})}{(1 +Ae^{i \theta})^{k+1}(1 +Be^{i \theta})^{2-k}}\right|\\
&\geq \frac{|\beta| m (A-B)(1-|A|)-(1+|A|)^{k}(1+|B|)^{3-k}}{(1 +|A|)^{k+1}(1 +|B|)^{2-k}}\\
&=:\phi(m)
\end{align*}
Observe that $\phi'(m)>0$ for $m \geq 1$. In view of above and  Lemma \eqref{lemma for admissibility}, simple computations gives the desired bound on $\beta$ in terms of $A$ and $B$.

\item When $k > 2$, consider
\begin{align*}
|\psi(r,s;z)|&=\left|\frac{1+Be^{i \theta}}{1+Ae^{i \theta}}-\beta \frac{m(A-B)e^{i\theta}(1+Be^{i \theta})^{k}}{(1 +Ae^{i \theta})^{k}(1 +Be^{i \theta})^2}\right|\\
&=\left|\frac{(1+Ae^{i\theta})^{k}(1+B e^{i \theta})-\beta m (A-B) e^{i \theta}(1+Ae^{i\theta})(1+Be^{i \theta})^{k-2}}{(1+Ae^{i \theta})^{k+1}}\right|\\
&\geq \frac{|\beta| m (A-B)(1-|A|)(1-|B|)^{k-2}-(1+|A|)^{k}(1+|B|)}{(1 +|A|)^{k+1}}\\
&=:\phi(m)
\end{align*}
Since, the function $\phi(m)$ is increasing for $m \geq 1$, similar computations as done in  case (i) gives the required result.
\end{enumerate}
\end{proof}

\begin{corollary}
Let $\beta_{0} \approx0.475319$, $f$ be an analytic function and $\mathcal{G}(z)$ be same as in Corollary \eqref{first corollary}. Then each of the following subordinations imply that $f \in \mathcal{S}^{*}[A,B]$.
\begin{enumerate}[(a)]
\item $\frac{f(z)}{zf'(z)}-\beta z \mathcal{G}(z) \prec \phi_{SG}$ for
$|\beta|(A-B)(1-|A|) \geq \beta_{0}(1+|A|)(1+|B|)^{2}+(1+|B|)^{3}$,
\item $\frac{f(z)}{zf'(z)}-\beta \frac{f(z)}{f'(z)} \mathcal{G}(z) \prec \phi_{SG}$ for
$|\beta|(A-B)(1-|A|) \geq \beta_{0}(1+|A|)^{2}(1+|B|)+(1+|A|)(1+|B|)^{2}$,
\item $\frac{f(z)}{zf'(z)}-\beta \left(\frac{f(z)}{f'(z)}\right)^2 \mathcal{G}(z) \prec \phi_{SG}$ for
$|\beta|(A-B)(1-|A|) \geq \beta_{0}(1+|A|)^{3}+(1+|A|)^{2}(1+|B|)$.
\end{enumerate}
\end{corollary}

\begin{theorem}
Suppose $-1 \leq B < A \leq 1$, $\gamma > 0$, $\beta_{0} \approx0.475319$  and $k$ be a non-negative integer.
Let $p$ be an analytic function satisfying the differential subordination
\[
p(z)+\frac{zp'(z)}{(\beta p(z)+ \gamma)^{k}} \prec \phi_{SG}(z) ,
\text { where }\]
\[
(A-B)(1-|B|) \geq \beta_{0}(1+|B|)^{2-k}(\beta(1+|A|)+\gamma (1+|B|))^{k}(2+|A|+|B|),
\text { when } 0 \leq k \leq 2,
\]
\[
(A-B)(1-|B|)^{k-1} \geq \beta_{0}
(2+|A|+|B|)(\beta(1+|A|)+\gamma (1+|B|))^{k},
\text { when } k>2.
\]

Then $p \in \mathcal{P}[A,B]$.
\end{theorem}

\begin{proof}
Let $\Omega$ be the domain as defined in Theorem\eqref{admissibility first thm}(a).
Define the function $\psi(r,s;z):\mathbb{C}\backslash \{0\} \times \mathbb{C} \times \mathbb{D} \rightarrow \mathbb{C}$ as
\[\psi(r, s;z) = r+\frac{s}{(\beta r + \gamma)^{k}}.\]
Then using equation\eqref{r and s}, the function $\psi$ becomes
\[\psi(r,s;z) = \frac{1+Ae^{i \theta}}{1+Be^{i \theta}}+\frac{m(A-B)(1+Be^{i \theta})^{k}e^{i \theta}}{(1+Be^{i \theta})^{2}(\beta(1+Ae^{i \theta})+\gamma (1+Be^{i \theta}))^{k}}\]
In view of Theorem \eqref{first order admissibility thm},  the desired subordination $p \prec (1+Az)/(1+Bz)$ will follow if we show that $\psi \in \Psi[\Omega;A,B]$.
For this, it suffices to show that
\[\left|\log \left(\frac{\psi(r,s;z)}{2-\psi(r,s;z)}\right)\right| \geq 1.\]

\begin{enumerate}[(i)]
\item When $0 \leq k \leq 2$, observe that
\begin{align*}
|\psi(r,s;z)|&=\left|\frac{1+Ae^{i \theta}}{1+Be^{i \theta}}+\frac{m(A-B)e^{i \theta}}{(1+Be^{i \theta})^{2-k}(\beta(1+Ae^{i \theta})+\gamma (1+Be^{i \theta}))^{k}}\right|\\
& =\left|\frac{\splitfrac{(1+Ae^{i \theta})(1+Be^{i \theta})^{2-k}(\beta(1+Ae^{i \theta})+\gamma (1+Be^{i \theta}))^{k}}{+m(A-B)e^{i \theta}(1+Be^{i \theta})}}{(1+Be^{i \theta})^{3-k}(\beta(1+Ae^{i \theta})+\gamma (1+Be^{i \theta}))^{k}}\right|\\
& \geq \frac{\splitfrac{m(A-B)|(1+B e^{i \theta})|-|(1+Ae^{i \theta})||(1+Be^{i \theta})^{2-k}|}{|(\beta(1+Ae^{i \theta})+\gamma (1+Be^{i \theta}))^{k}|}}{|(1+Be^{i \theta})^{3-k}||(\beta(1+Ae^{i \theta})+\gamma (1+Be^{i \theta}))^{k}|}\\
& \geq \frac{m (A-B)(1-|B|)-(1+|A|)(1+|B|)^{2-k}|(\beta(1+Ae^{i \theta})+\gamma (1+Be^{i \theta}))^{k}|}{(1+|B|)^{3-k}(\beta(1+|A|)+\gamma (1+|B|))^{k}}
\end{align*}
Similar analysis as done in Theorem \eqref{admissibility first thm}(a) gives that $\psi(r,s;z) \notin \Omega$ for
\[(A-B)(1-|B|) \geq \beta_{0}(1+|B|)^{2-k}(\beta(1+|A|)+\gamma (1+|B|))^{k}(2+|A|+|B|).\]

\item When $k > 2$, consider
\begin{align*}
|\psi(r,s;z)|&=\left|\frac{1+Ae^{i \theta}}{1+Be^{i \theta}}+\frac{m(A-B)e^{i \theta}(1+Be^{i \theta})^{k-2}}{(\beta(1+Ae^{i \theta})+\gamma (1+Be^{i \theta}))^{k}}\right|\\
&=\left|\frac{(1+Ae^{i \theta})(\beta(1+Ae^{i \theta})+\gamma (1+Be^{i \theta}))^{k}+m(A-B)e^{i \theta}(1+Be^{i \theta})^{k-1}}{(1+Be^{i \theta})(\beta(1+Ae^{i \theta})+\gamma (1+Be^{i \theta}))^{k}}\right|\\
& \geq \frac{m(A-B)|(1+B e^{i \theta})^{k-1}|-|(1+Ae^{i \theta})||(\beta(1+Ae^{i \theta})+\gamma (1+Be^{i \theta}))^{k}|}{|(1+Be^{i \theta})||(\beta(1+Ae^{i \theta})+\gamma (1+Be^{i \theta}))^{k}|}\\
& \geq \frac{m (A-B)(1-|B|)^{k-1}-(1+|A|)|(\beta(1+Ae^{i \theta})+\gamma (1+Be^{i \theta}))^{k}|}{(1+|B|)|(\beta(1+|A|)+\gamma (1+|B|))^{k}|}
\end{align*}
On the similar lines as in proof of part (i), we get the desired result.
\end{enumerate}
\end{proof}

\section{Second order differential subordination}
In this section, sufficient conditions are obtained so that the subordination implication \[p(z) \prec \frac{1+Az}{1+Bz}\] holds whenever
$\psi(p(z),zp'(z),z^2p''(z);z)$ is subordinate to Modified Sigmoid function, exponential function and Janowski function.

\begin{theorem}
Let $-1<B<A<1$, $\gamma>0$, $\beta >0$ and $\beta_{0} \approx 0.475319$. Let $p$ be an analytic function satisfying $p(0)=1$.
Then, each of the following is sufficient for $p \in \mathcal{P}[A,B]$. \begin{align*}
(a) & 1+\gamma z p^{\prime}(z)+ \beta z^2 p''(z) \prec \phi_{SG}(z), \text{where}\\
& (A-B)[\gamma(1+B^2+2B)-2B\beta (B+1)] \geq (\beta_{0}+1)(1+B^2+2B)^2 \text{ for } B>0 \text{ and }\\
& (A-B)[\gamma(1+B^2-2B)-2B\beta (B-1)] \geq (\beta_{0}+1)(1+B^2-2B)^2 \text{ for } B<0.
\end{align*}
\begin{align*}
(b) & 1+\gamma z p^{\prime}(z)+ \beta z^2 p''(z) \prec e^{z}, \text{where}\\
& (A-B)[\gamma(1+B^2+2B)-2B\beta (B+1)] \geq (e-1)(1+B^2+2B)^2 \text{ for } B>0 \text{ and }\\
& (A-B)[\gamma(1+B^2-2B)-2B\beta (B-1)] \geq (e-1)(1+B^2-2B)^2 \text{ for } B<0.
\end{align*}
\begin{align*}
(c) & 1+\gamma z p^{\prime}(z)+ \beta z^2 p''(z) \prec (1+Az)/(1+Bz), \text{where}\\
& (A-B)(1-B^{2})[\gamma(1+B^2+2B)-2B\beta (B+1)]-|B|(A-B)(1+B^2+2B)^2 \\ & \geq(A-B)(1+B^2+2B)^2, \text{ for } B>0 \text{ and }\\
& (A-B)(1-B^{2})[\gamma(1+B^2-2B)-2B\beta (B-1)]-|B|(A-B)(1+B^2-2B)^2 \\ & \geq(A-B)(1+B^2-2B)^2,  \text{ for } B<0.
\end{align*}
\end{theorem}

\begin{proof}
\begin{enumerate}[(a)]
\item
Let $\Omega = \phi_{SG}(\mathbb{D}) = \{w \in \mathbb{C}: \left|\log \left({w}/{(2-w)}\right)\right| < 1\}$.
Consider the analytic function
$\psi : \mathbb{C}^3 \times \mathbb{D} \rightarrow \mathbb{C}$
defined as
\[\psi(r,s,t;z)=1+\gamma s + \beta t\]
For $\psi \in \Psi[\Omega;A,B]$, we must have $\psi(r,s,t;z) \notin \Omega$.
By Theorem \eqref{second order admissibility theorem}, this implication is true if\begin{equation}\label{log inequality second order}
\left|\log \left(\frac{\psi(r,s,t;z)}{2-\psi(r,s,t;z)}\right)\right| \geq 1.
\end{equation}

By Lemma \eqref{lemma for admissibility}, the inequality \eqref{log inequality second order} holds if and only if $|\psi(r,s,t;z)| \geq \beta_{0}$.
A calculation shows that
\begin{align*}
|\psi(r,s,t;z)|&=|1+\gamma s + \beta t|\\
& \geq 1+ \gamma |s| \left|1+\left(\frac{\beta}{\gamma}\right) \frac{t}{s} \right|-1\\
& \geq \gamma |s| \operatorname{Re}\left(1+\left(\frac{\beta}{\gamma}\right) \frac{t}{s}\right)-1\\
& \geq m \gamma d(\theta) \left(1+\left(\frac{\beta}{\gamma}\right)(m g(\theta)+m-1)\right)-1\\
&\geq \begin{cases} \frac{m(A-B)}{1+B^2+2B}\left(\gamma+ \frac{-2B(B+1)\beta m}{1+B^2+2B}\right)-1,  B>0\\
\frac{m(A-B)}{1+B^2-2B} \left(\gamma+ \frac{-2B(B-1)\beta m}{1+B^2-2B}\right)-1, B<0
\end{cases}\\
&:=\phi(m),
\end{align*}
where $d(\theta)$ and
$g(\theta)$ are given by \eqref{d(theta)} and \eqref{g(theta)} respectively. Observe that $\phi(m)$ is increasing function for $m\geq1$.
Therefore, we have $|\psi(r,s,t;z)| \geq \phi(1)\geq\beta_{0}$ and hence, $\psi \in \Psi[\Omega;A,B]$. By Theorem \ref{second order admissibility theorem}, $p(z) \prec (1+Az)/(1+Bz)$.

\item Consider the domain $\Omega = \{w \in \mathbb{C}: |\log w| < 1\}$.
Let the function $\psi:\mathbb{C}^{3} \times \mathbb{D} \rightarrow \mathbb{C}$ be defined as
\[\psi(r,s,t;z) = 1+\gamma s + \beta t.\]
For $\psi \in \Psi[\Omega;A,B]$, we must have
$\psi(r,s,t;z) \notin \Omega$.
In order to satisfy this relation, it is sufficient to show that
\[|\log(\psi(r,s,t;z))| \geq 1.\]
Since
\begin{align*}
\left| \gamma s \left( 1+ \frac{\beta t}{\gamma s}\right) \right|
& \geq \gamma |s| \left |1+\left(\frac{\beta}{\gamma}\right) \frac{t}{s} \right|\\
& \geq \gamma |s| \operatorname{Re}\left(1+\left(\frac{\beta}{\gamma}\right) \frac{t}{s}\right)\\
& \geq m \gamma d(\theta) \left(1+\left(\frac{\beta}{\gamma}\right)(m g(\theta)+m-1)\right) \\
&\geq \begin{cases} \frac{m(A-B)}{1+B^2+2B}\left(\gamma- \frac{2B(B+1)\beta m}{1+B^2+2B}\right), B>0\\\frac{m(A-B)}{1+B^2-2B} \left(\gamma- \frac{2B(B-1)\beta m}{1+B^2-2B}\right), B<0
\end{cases}\\
&:=\phi(m)
\end{align*}
and $\phi(m)$ is increasing function of $\phi$, by Lemma \ref{lemma for e^{z}},
$\left| \gamma s \left( 1+ \frac{\beta t}{\gamma s}\right) \right| \geq e-1$. Thus, by Theorem \ref{second order admissibility theorem}, $p(z) \prec (1+Az)/(1+Bz)$.

\item Consider the domain
\[\Omega=\left\{w \in \mathbb{C} : \left|w - \frac{1-AB}{1-B^{2}}\right|  < \frac{A-B}{1-B^{2}}\right\}.\]
Let
$\psi : \mathbb{C}^3 \times \mathbb{D} \rightarrow \mathbb{C}$
be defined as
$\psi(r,s,t;z) = 1+\gamma s + \beta t$.
Now, $\psi \in \Psi[\Omega;A,B]$, if $\psi(r,s,t;z) \notin \Omega$.
On the similar lines on the proof of part(a),
\begin{align*}
\left|\psi(r,s,t;z) - \frac{1-AB}{1-B^{2}}\right| & = \left|1+\gamma s + \beta t - \frac{1-AB}{1-B^{2}} \right|\\
& \geq \gamma |s| \operatorname{Re}\left(1+\left(\frac{\beta}{\gamma}\right) \frac{t}{s}\right)-\frac{|B|(A-B)}{1-B^{2}}\\
& \geq m \gamma d(\theta) \left(1+\left(\frac{\beta}{\gamma}\right)(m g(\theta)+m-1)\right)-\frac{|B|(A-B)}{1-B^{2}}\\
:=\phi(m).
\end{align*}
 Using the values of $d(\theta)$ and $g(\theta)$ as given in the equations \eqref{d(theta)} and \eqref{g(theta)}, and first derivative test for function  $\phi$ we have,

for $B>0$,
\begin{align*}
\left|\psi(r,s,t;z) - \frac{1-AB}{1-B^{2}}\right|
& \geq \frac{(A-B)}{1+B^2+2B}\left(\gamma - \frac{2B(B+1)\beta }{1+B^2+2B}\right) - \frac{|B|(A-B)}{1-B^{2}}\\
& \geq \frac{A-B}{1-B^{2}},
\end{align*}
and for $B<0$,
\begin{align*}
	\left|\psi(r,s,t;z) - \frac{1-AB}{1-B^{2}}\right|
	& \geq \frac{(A-B)}{1+B^2-2B}\left(\gamma -  \frac{2B(B-1)\beta }{1+B^2-2B}\right) - \frac{|B|(A-B)}{1-B^{2}}\\
	& \geq \frac{A-B}{1-B^{2}}.
\end{align*}
Therefore, $\psi \in \Psi[\Omega;A,B]$ and hence,  by Theorem \ref{second order admissibility theorem}, $p(z) \prec (1+Az)/(1+Bz)$.\qedhere
\end{enumerate}
\end{proof}

\begin{corollary}
Let $\gamma$ and $\beta$ be positive integers and $f$ be an analytic function. Set
\begin{align*}
\mathcal{H}(z) = & 1+ \gamma \Bigg(\frac{z^{2}f''(z)}{f(z)} - \Bigg( \frac{zf'(z)}{f(z)}\Bigg)^2 +\frac{zf'(z)}{f(z)}\Bigg) + \beta \Bigg( \frac{z^{3}f'''(z)}{f(z)} \\
& + \frac{2z^{2}f''(z)}{f(z)} + 2 \Bigg(\frac{zf'(z)}{f(z)}\Bigg)^{3} - 2\Bigg( \frac{zf'(z)}{f(z)} \Bigg)^{2} - \frac{3z^{3}f'(z)f''(z)}{f(z)^{2}} \Bigg).
\end{align*}
Then, $f \in \mathcal{S}^{*}[A,B]$ if any one of the following condition hold.
\begin{align*}
(a) & \mathcal{H}(z) \prec \phi_{SG}(z), \text{where} \\
 &(A-B)[\gamma(1+B^2+2B)-2B\beta  (B+1)] \geq (\beta_{0}+1)(1+B^2+2B)^2 \text{ for } B>0 \text{ and } \\
& (A-B)[\gamma(1+B^2-2B)-2B\beta  (B-1)] \geq (\beta_{0}+1)(1+B^2-2B)^2 \text{ for } B<0.
\end{align*}
\begin{align*}
(b) &\mathcal{H}(z) \prec e^{z}, \text{where} \\
&(A-B)[\gamma(1+B^2+2B)-2B\beta  (B+1)] \geq (e-1)(1+B^2+2B)^2 \text{ for } B>0 \text{ and }  \\
& (A-B)[\gamma(1+B^2-2B)-2B\beta  (B-1)] \geq (e-1)(1+B^2-2B)^2 \text{ for } B<0.
\end{align*}
\begin{align*}
(c) & \mathcal{H}(z) \prec (1+Az)/(1+Bz), \text{where}\\
& (A-B)(1-B^{2})[\gamma(1+B^2+2B)-2B\beta (B+1)]-|B|(A-B)(1+B^2+2B)^2 \\ & \geq(A-B)(1+B^2+2B)^2, \text{ for } B>0 \text{ and }\\
& (A-B)(1-B^{2})[\gamma(1+B^2-2B)-2B\beta (B-1)]-|B|(A-B)(1+B^2-2B)^2 \\ & \geq(A-B)(1+B^2-2B)^2,  \text{ for } B<0.
\end{align*}
\end{corollary}

\begin{theorem}
Suppose $-1<B<A<1$, $\beta_{0} \approx 0.475319$,  $\beta >0$ and $\gamma >0$. Let $p$ be an analytic function which satisfies the condition $p(0)=1$ and the following inequalities holds:
\[
(A-B)(1+B)[\gamma(1+B^2+2B)-2B(B+1)\beta ]-(1+A)(1+B^2+2B)^2 \geq \beta_{0} (1+B) (1+B^2+2B)^2
\]
for $B>0$, and
\[
(A-B)(1+B)[\gamma(1+B^2-2B)-2B(B-1)\beta ]-(1+A)(1+B^2-2B)^2 \geq \beta_{0} (1+B) (1+B^2-2B)^2
\]
for $B<0$. Then,
\[p(z) +\gamma z p^{\prime}(z)+ \beta z^2 p''(z) \prec \phi_{SG}(z)\]
implies \[p \prec \frac{1+Az}{1+Bz}.\]
\end{theorem}

\begin{proof}
Let $\Omega$ be the domain defined in Theorem \eqref{admissibility first thm}(a).
Consider the function
$\psi : \mathbb{C}^3 \times \mathbb{D} \rightarrow \mathbb{C}$
defined as
\[\psi(r,s,t;z) = r +\gamma s + \beta t.\]
For $\psi$ to be in $\Psi[\Omega;A,B]$, we must have $\psi(r,s,t;z) \notin \Omega$.
By Theorem \eqref{second order admissibility theorem}, this result is true if\[\left|\log \left(\frac{\psi(r,s,t;z)}{2-\psi(r,s,t;z)}\right)\right| \geq 1\]
Using Lemma \eqref{lemma for admissibility}, this inequality holds if and only if
\begin{equation}\label{condition using beta_{0}}
|\psi(r,s,t;z)| \geq \beta_{0}
\end{equation}
where $\beta_{0}$ is the positive real root of the equation in \eqref{lemma inequality}.
If $k(\theta)$, $g(\theta)$ and $d(\theta)$ are given by the equations \eqref{k(theta)}, \eqref{g(theta)} and \eqref{d(theta)} respectively, then
\begin{align*}
|\psi(r,s,t;z)| & =|r + \gamma s + \beta t|\\
& \geq  \gamma |s| \left |1+\left(\frac{\beta}{\gamma}\right) \frac{t}{s} \right|-|r|\\
& \geq  \gamma |s| \operatorname{Re}\left(1+\left(\frac{\beta}{\gamma}\right) \frac{t}{s}\right)-|r|\\
& \geq  m \gamma d(\theta) \left(1 + \left(\frac{\beta}{\gamma}\right)(m g(\theta)+m-1)\right)-|k(\theta)|\\
&\geq  m \gamma d(\theta) \left(1 + \left(\frac{\beta}{\gamma}\right)(m g(\theta)+m-1)\right)-\frac{1+A}{1+B}\\
&:=\phi(m).
\end{align*}
Since $\phi(m)$ is increasing function, we have,
 for $B>0$,
\[|\psi(r,s,t;z)| \geq  \frac{(A-B)}{1+B^2+2B}\left(\gamma  - \frac{2B(B+1)\beta }{1+B^2+2B}\right)-\frac{1+A}{1+B} \text{ and }\]
for $B<0$,
\[|\psi(r,s,t;z)| \geq  \frac{(A-B)}{1+B^2-2B}\left(\gamma -  \frac{2B(B-1)\beta }{1+B^2-2B}\right)-\frac{1+A}{1+B}.\]
Therefore, inequality \eqref{condition using beta_{0}} is satisfied and hence, we get the required result.
\end{proof}

\end{document}